\pdfoutput=1

\documentclass{amsart}

\usepackage[utf8]{inputenc}
\usepackage[T1]{fontenc}
\usepackage[english]{babel}
\usepackage{lmodern} 

\usepackage{amsfonts}
\usepackage{cases}
\usepackage{mathrsfs}
\usepackage{bbm}
\usepackage{amssymb}
\usepackage{txfonts}
\usepackage{amscd}
\usepackage{amsfonts,latexsym,amsmath,amsthm,amsxtra,mathdots}
\usepackage[bookmarks,colorlinks]{hyperref}
\usepackage[all,cmtip]{xy}
\RequirePackage{amsmath} \RequirePackage{amssymb}
\usepackage{color}
\usepackage{colordvi}
\usepackage{multicol}
\usepackage{tikz}
\usepackage{hyperref}
\usepackage{cleveref}
\usepackage{graphicx}
\usepackage{amsmath}
\usepackage{amsmath,amscd}
\usepackage[normalem]{ulem}
\usetikzlibrary{matrix,arrows}
\usepackage{textcomp}
\usepackage{tikz-cd}
\usepackage{nicefrac}
\usepackage[shortlabels]{enumitem}

\newcommand{\id}{{{\rm id}}}

\newcommand{\subgroup}{\leq}

    \newcommand{\BE}{{\mathbb {E}}}

     \newcommand{\BN}{{\mathbb {N}}}

    \newcommand{\BR}{{\mathbb {R}}}
    \newcommand{\R}{{\mathbb {R}}}

    \newcommand{\BZ}{{\mathbb {Z}}}
    \newcommand{\Z}{{\mathbb {Z}}}
    \newcommand{\Q}{{\mathbb {Q}}}

    \newcommand{\Hom}{{\mathrm{Hom}}}

    \newcommand{\Tor}{{\mathrm{Tor}}}

\newcommand{\cd}{\operatorname{cd}}
\newcommand{\vcd}{\operatorname{vcd}}

\newcommand{\fingd}[1]{\ensuremath{\underline{\mathrm{gd}} \, #1}}

\newcommand{\vcycgd}[1]{\ensuremath{\underline{\underline{\mathrm{gd}}} \, #1}}

\newcommand{\catname}[1]{{\normalfont\textbf{#1}}}

\newcommand{\Or}{\catname{Or}}

\newcommand{\gd}{\operatorname{gd}}

\newcommand{\define}{\emph}

\newcommand{\CondM}{\operatorname{M}}
\newcommand{\CondNM}{\operatorname{NM}}

\newcommand{\epiarrow}{\twoheadrightarrow}

\newcommand{\abel}{\operatorname{ab}}

\newcommand{\eqnstop}{\, .}
\newcommand{\eqncomma}{\, ,}

\def\-{^{-1}}

\DeclareRobustCommand{\huu}[1]{\texorpdfstring{\uu{#1}}{#1}}

\newcommand{\delete}[1]{}

    \theoremstyle{plain}

\newtheorem{thm}{Theorem}[section]
\newtheorem{defn}[thm]{Definition} 
\newtheorem{lem}[thm]{Lemma}
\newtheorem{prop}[thm]{Proposition}
\newtheorem{cor}[thm]{Corollary}
\newtheorem{rem}[thm]{Remark}

\newtheorem{question}[thm]{Question}

\newtheorem{notation}[thm]{Notation}

\newtheorem*{thm*}{Theorem}

\newtheorem*{rem*}{Remark}

\newtheorem*{cor*}{Corollary}

\newtheorem*{namedthm}{\namedthmname}
\newcounter{namedthm}

\renewenvironment{proof}[1][\proofname]{{\bfseries #1. \quad}}{\qed}

\makeatletter
\newenvironment{named}[1]
  {\def\namedthmname{#1}%
   \refstepcounter{namedthm}%
   \namedthm\def\@currentlabel{#1}}
  {\endnamedthm}
\makeatother

\newcommand{\family}[1]{\mathcal{#1}}
\newcommand{\TR}{\family{TR}}
\newcommand{\FIN}{\family{F}in}
\newcommand{\CYC}{\family{C}yc}
\newcommand{\VCYC}{\family{VC}yc}

    \numberwithin{equation}{section}

\newcommand{\uueg}{{\underline{\underline E}}G}
\newcommand{\uubg}{{\underline{\underline B}}G}

\newcommand{\uu}[1]{\uline{\uline{#1}}}
\renewcommand{\u}[1]{\uline{#1}}

\def\Proof{\noindent{\bf Proof}\quad}
\def\qed{\hfill$\square$\smallskip}

\newenvironment{cond}[1]{%
  \vspace{1.2ex}\begin{enumerate}%
  \renewcommand*{\theenumi}{#1}%
  \renewcommand*{\labelenumi}{#1}%
  \item%
}{%
  \end{enumerate}\vspace{1.2ex}%
}

\usepackage[
    backend=biber,
    style=alphabetic,
    maxbibnames=99
]{biblatex}
  
\begin{document}

\title{Some results related to finiteness properties of groups for families of  subgroups}

\author{Timm von Puttkamer}
\address{University of Bonn, Mathematical Institute, Endenicher Allee 60, 53115 Bonn, Germany}
\email{tvp@math.uni-bonn.de}

\author{Xiaolei Wu}
\address{University of Bonn, Mathematical Institute, Endenicher Allee 60, 53115 Bonn, Germany}
\email{xwu@math.uni-bonn.de}

\date{July, 2018}


\begin{abstract}
Let $\uueg$ be the classifying space of $G$ for the family of virtually cyclic subgroups. We show that an Artin group admits a finite model for $\uu E G$ if and only if it is virtually cyclic. This solves a conjecture of Juan-Pineda--Leary and a question of L\"uck-Reich-Rognes-Varisco for Artin groups. We then study conjugacy growth of CAT(0) groups and show that if a CAT(0) group contains a free abelian group of rank two, its conjugacy growth is strictly faster than linear. This also yields an alternative proof for the fact that a CAT(0) cube group admits a finite model for $\uu E G$ if and only if it is virtually cyclic. Our last result deals with the homotopy type of the quotient space $\uubg =\uueg/G$. We show for a poly-$\BZ$-group $G$, that $\uubg$ is homotopy equivalent to a finite CW-complex if and only if $G$ is  cyclic. 
\end{abstract}

\maketitle

\section{Introduction}

In this paper, we continue our study \cite{vPW2,vPW1} of finiteness properties of classifying spaces for families of subgroups. Given a discrete group $G$, a family $\mathcal{F}$ of subgroups of $G$ is a set of subgroups of $G$ which is closed under conjugation and taking subgroups. The classifying space $E_{\mathcal{F}}(G)$ is a $G$-CW-complex characterized by the following property: the fixed point set $E_{\mathcal{F}}(G)^H$ is contractible for any $H \in \mathcal{F}$ and empty otherwise. Recall that a $G$-CW-complex $X$ is said to be finite (resp. of finite type) if it has finitely many orbits of cells (resp. finitely many orbits of cells in each dimension). We abbreviate $E_{\mathcal{F}}(G)$ by $\uueg$ for $\mathcal{F} = \VCYC$ the family of virtually cyclic subgroups, $\u EG$ for $\mathcal{F} = \FIN$ the family of finite subgroups and $EG$ for $\mathcal{F}$ the family consisting only of the trivial subgroup. In \cite[Conjecture 1]{JuanPinedaLeary}, Juan-Pineda and Leary formulated the following conjecture:

\begin{named}{Conjecture A}   \cite [Juan-Pineda and Leary]{JuanPinedaLeary}\label{ConjA}
Let $G$ be a group admitting a finite model for $\uueg$. Then $G$ is virtually cyclic.
\end{named}

This conjecture has been proven for a rather large class of groups  \cite{GW2013,KMN2011,vPW2,vPW1} (see Theorem \ref{bVCYC-thm} and the remark thereafter for a summary). Note that a group $G$ admits a model for $\uueg$ with a finite $0$-skeleton if and only if the following property holds

\begin{cond}{(BVC)}\label{cond:BVC}
\textit{$G$ has a finite set of virtually cyclic subgroups $\{V_1,V_2,\ldots, V_n\}$ such that every virtually cyclic subgroup of $G$ is conjugate to a subgroup of some $V_i$.}
\end{cond}

Following Groves and Wilson, we shall call this property BVC. So far almost all of the proofs of \ref{ConjA} boil down to verifying whether a given group has BVC. Our first theorem extends the existing results to Artin groups.

\begin{named}{Theorem A}\label{thma}
An Artin group has BVC if and only if it is virtually cyclic. 
\end{named}

For the proof, we used a weaker property compared to BVC, namely $b\VCYC$ (see Definition \ref{defn-bfamily} and remarks below it). It is a related notion designed to be closed under forming quotients, see \Cref{bfquotient}. Let $\CYC$ be the family of cyclic subgroups. A related question of L\"uck-Reich-Rognes-Varisco asks: is it true that a group $G$ has a model of finite type for $E_{\CYC} G$ if and only if $G$ is finite, cyclic or infinite dihedral? By \cite[Lemma 3.2 + Corollary 3.10]{vPW2}, our result also answers their question for Artin groups.

In \cite{vPW2}, we established some connection between the BVC property and the conjugacy growth invariant. In particular, we showed \cite[Corollary 2.5]{vPW2} that if a semihyperbolic group has BVC, then it has at most linear conjugacy growth. But we could not determine exactly when a semihyperbolic group has at most linear conjugacy growth \cite[Remark 2.6]{vPW1}. Instead we give some partial answer to this question here.

\begin{named}{Theorem B}\label{thmb}
Let $G$ be a  CAT(0) group containing $\BZ^2$ as a subgroup, then the conjugacy growth of $G$ is strictly faster than linear. If $G$ is a CAT(0) cube group, then it has at most linear conjugacy growth if and only it is virtually cyclic.
\end{named}

Given a group $G$ and a family of subgroups $\family F$ of $G$ we also study the finiteness properties of the classifying space $B_{\family F}(G) = E_{\family F}(G)/G$. The following question goes back to \cite[Remark 17]{JuanPinedaLeary} and motivated our study:
\begin{named}{Question C}\label{uubg-finite-ques}
Suppose $\uu BG$ is homotopy equivalent to a finite CW-complex. Is $\uu BG$ necessarily contractible?
\end{named}
In contrast to \ref{uubg-finite-ques}, in the case of the family of finite subgroups, Leary and Nucinkis showed \cite{LearyNucinkis2001} that every connected finite CW-complex is homotopy equivalent to $\underline{B}G$ for some group $G$. In fact, Januszkiewicz and \'Swiçatkowski \cite[Theorem M]{JJ06} further proved that one can take this group to be hyperbolic. On the other hand, Juan-Pineda and Leary showed that if $G$ is hyperbolic, then $\uubg$ is homotopy equivalent to a finite CW-complex if and only if $G$ is virtually cyclic \cite[Corollary 16]{JuanPinedaLeary}. We extend their result to the following: 

\begin{named}{Theorem C}
Suppose $G$ is an abelian group or a poly-$\Z$-group. Then $\uu BG$ is homotopy equivalent to a finite CW-complex if and only if $G$ is locally virtually cyclic.
\end{named}

Since the torsion elements in a finitely generated nilpotent group form a finite normal subgroup and the quotient by this normal subgroup is a poly-$\Z$-group, we have the following.

\begin{named}{Corollary D}
Suppose $G$ is a finitely generated nilpotent group. Then $\uu BG$ is homotopy equivalent to a finite CW-complex if and only if $G$ is  virtually cyclic.
\end{named}

Parts of this paper have previously appeared in the first author's thesis.

\textbf{Acknowledgements.} The first author was supported by an IMPRS scholarship of the Max Planck Society. The second author was partially supported by Prof. Wolfgang L\"uck's ERC Advanced Grant ``KL2MG-interactions'' (no.  662400). Part of this work was carried out when the second author was a postdoc at the Max Planck Institute for  Mathematics in Bonn. We also want to thank Pieter Moree for drawing our attention to results of Paul Bernays and himself.

\section{Review on Groups Admitting a Finite Model for \texorpdfstring{$\uueg$}{EG} and Property \texorpdfstring{$b\mathcal{F}$}{bF}} \label{basic}
In this section we first review some properties  and results on groups admitting a finite model for $\uueg$. Most of this material is taken directly from \cite[Section 1]{vPW2}.

We summarize the properties of groups admitting a finite model for $\uueg$ as follows:

\begin{prop}
Let $G$ be a group admitting a finite model for $\uueg$, then

\begin{enumerate}
    \item  $G$ has BVC, 

    \item $G$ admits a finite model for $\underline{E}G$. 

    \item For every finite subgroup $H \subgroup G$, the Weyl group $W_GH$ is finitely presented and of type $FP_{\infty}$. Here $W_GH = N_G(H)/H$, where $N_G(H)$ is the normalizer of $H$ in $G$.

    \item $G$ admits a model of finite type for $EG$. In particular, $G$ is finitely presented.

\end{enumerate}

\end{prop}

\begin{rem}
If one replaces finite by finite type in the assumptions of the above proposition, then the conclusions still hold if one also replaces finite by finite type in (b).
\end{rem}

\begin{lem} \cite[Lemma 1.3]{vPW1} \label{BVCfinitezeroskeleton}
Let $G$ be a group. There is a model for $\uueg$ with finite $0$-skeleton if and only if $G$ has BVC.
\end{lem}

The following structure theorem about virtually cyclic groups is well known, see for example \cite[Proposition 4]{JuanPinedaLeary} for a proof.

\begin{lem}  \label{vcstru}
Let $G$ be a virtually cyclic group. Then $G$ contains a unique maximal normal finite subgroup $F$ such that one of the following holds

\begin{enumerate}

\item the finite case, $G =F$;
\item the orientable case, $G/F$ is the infinite cyclic group;
\item the nonorientable case, $G/F$ is the infinite dihedral group.

\end{enumerate}

\end{lem}

\begin{lem} \label{BVCfinitesubgroup}\cite[Lemma 1.7]{vPW1}
If a group $G$ has BVC, then $G$ has finitely many conjugacy classes of finite subgroups. In particular, the order of finite subgroups in $G$ is bounded.
\end{lem}

\begin{lem} \cite[Lemma 5.6]{KMN2011} \label{finiteindexsubgroupbvc}
If a group $G$ has BVC, then any finite index subgroup also has BVC.
\end{lem}

The following notion was introduced by the authors in \cite[Definition 1.10]{vPW2}
\begin{defn} \label{defn-bfamily}
Let $\family F$ be a family of subgroups. For a natural number $n \geq 1$, we say that a group $G$ has property $n \mathcal F$ if there are $H_1, \ldots, H_n \in \mathcal F$ such that any cyclic subgroup of $G$ is contained in a conjugate of $H_i$ for some $i$. We say that $G$ has $b \family{F}$ if $G$ has $n \family{F}$ for some $n \in \BN$. 
\end{defn}

We are mostly interested in $b \VCYC$ as both $b\CYC$ and BVC imply $b \VCYC$, which leads to unified proofs when we deal with finiteness properties of $E_{\CYC}(G)$ and $\uu EG$. The three notions  generally do not coincide \cite[Example 1.11,1.12]{vPW2}.

\begin{lem}\label{bfquotient}\cite[Lemma 1.13]{vPW2}
Suppose the family $\family{F}$ is closed under quotients. If $\pi \colon G \to Q$ is an epimorphism and $G$ has $b \family{F}$, then $Q$ has $b \family{F}$.
\end{lem}

\begin{lem}\label{finite-index-bf}\cite[Lemma 1.14]{vPW2}
Let $K \subgroup G$ be a finite index subgroup and suppose $G$ has $b \family F$. Then $K$ also has $b \family F$.
\end{lem}

\begin{thm}\label{bVCYC-thm}
Let $G$ be a finitely generated group in one of the following classes
\begin{enumerate}
    \item \label{bVCYC-solvable} virtually solvable groups,
    \item \label{bVCYC-one-relator} one-relator groups,
    \item \label{bVCYC-acy-hyperbolic} acylindrically hyperbolic groups,
    \item 3-manifold groups,
    \item CAT(0) cube groups,
    \item \label{linear-bvc} linear groups,
    \item \label{npc-manifold} groups acting properly and discontinuously on a Hadamard manifold via isometries such that the quotient has finite volume.
\end{enumerate}
If $G$ has $b \VCYC$, then $G$ is virtually cyclic. 
\end{thm}

\Proof We only need to prove the last item \cite[Theorem 1.15, Theorem 2.11]{vPW2}. This follows from the rank rigidity theorem \cite[Theorem C,~p.6]{Ba95} for non-positively curved manifolds. In fact, let $M$ be the Hadamard manifold described in (g). By \cite[Theorem C, ~p.6]{Ba95}, we have that either $M$ is rank one or a symmetric space, or a non-trivial Riemannian product. If $M$ is rank one, by \cite[Theorem B(iii),~p.5]{Ba95}, $G$ contains a rank one isometry, hence $G$ does not have $b\VCYC$ as it is acylindrically hyperbolic \cite{Sis15}. If $M$ is a symmetric space, since $G$ acts properly on $M$, we have that $G$ surjects onto a linear group with finite kernel. But a linear group has $b\VCYC$ if and only if it is virtually cyclic. Hence if $G$ has $b \VCYC$, it is virtually cyclic. Now we assume that $M$ is a Riemannian product. If $M$ is a product of symmetric spaces, then we are again done by the same reasoning as above. Thus, by rank rigidity, we can assume that $M = M_1\times M_2\times \cdots \times M_n\times S$, such that each $M_i$ is of rank one and $S$ is a product of symmetric spaces.  Now up to replacing $G$ by a finite index subgroup, we can assume that $G$ preserves the product. Note that $G$ still has $b \VCYC$ by \Cref{finite-index-bf}. Now $G$ maps to the isometry group of $M_1$, and we denote the image by $G_1$. Note that $G_1$ acts on $M_1$ properly and discontinuously with finite covolume. Applying \cite[Theorem B(iii),p.5]{Ba95} again, we have that $G_1$ contains a rank one isometry. Thus $G_1$ does not have $b\VCYC$. By \Cref{bfquotient}, $G$ also does not have $b\VCYC$.

\qed

\begin{rem}\label{EA-rem}
\ref{ConjA} is also known for elementary amenable groups \cite[Corollary 5.8]{KMN2011}, but we do not know whether an  elementary amenable group has $b \VCYC$ (or BVC) if and only if it is virtually cyclic.
\end{rem}

\begin{cor}\label{BVChomology} If $G$ has $b\VCYC$ and surjects onto a finitely generated group $Q$ that lies in one of the classes described in \Cref{bVCYC-thm}, then $Q$ is virtually cyclic. In particular, the abelianization $H_1(G, \BZ)$ is finitely generated of rank at most one.
\end{cor}

\section{Artin groups}

In this section, we prove that an Artin group has $b \VCYC$ if and only if it is virtually cyclic.

\subsection{Basics about Artin groups}

We give a quick definition of Artin groups, see \cite{Mc} for more information about Artin groups. An Artin group (or generalized braid group) $A$ is a group with a presentation of the form
$$\langle x_1,x_2,\cdots, x_n \mid \langle x_1,x_2\rangle^{m_{1,2}} =\langle x_2,x_1\rangle^{m_{2,1}},\cdots,\langle x_{n-1},x_{n}\rangle^{m_{n-1,n}} =\langle x_{n},x_{n-1}\rangle^{m_{n,n-1}} \rangle \eqncomma $$

where $m_{i,j}=m_{j,i}\in\{2,3,\cdots,\infty\},i<j$, and for $m_{i,j}\in\{2,3,\cdots\}$, $\langle x_i,x_j\rangle^{m_{i,j}}$ denotes an alternating product of $x_i$ and $x_j$ of length $m$, beginning with $x_i$. For example $\langle x_1,x_3\rangle^3 =x_1x_3x_1$. When $m_{i,j}=\infty$, there is (by convention) no relation for $x_i$ and $x_j$.

This data can be encoded by a Coxeter diagram; this is a labelled graph with $n$ vertices
$v_1,\cdots, v_n$, where two vertices $v_i$ and $v_j$ are connected by an edge if $m_{i,j}\geq  3$ and edges
are labeled by $m_{i,j}$ whenever $m_{i,j} \geq 4$.

Given an Artin group with the above presentation, one further obtains a Coxeter group $C$ by adding the relation $x_i^2 =1$ for all $i$. We say that the Artin group $A$ is of spherical type if
the associated Coxeter group $W$ is finite. When the Coxeter diagram associated to the Artin group is connected, the groups $A$ and $W$ are said to be irreducible, otherwise reducible.

\subsection{The $b\VCYC$ property for Artin groups}

We first deal with the case when the Coxeter group corresponding to the Artin group is virtually cyclic. 

\begin{prop}
Let $A$ be an Artin group and let $W$ be its corresponding Coxeter group. If $W$ is infinite virtually cyclic, then $A$ does not have $b \VCYC$ .
\end{prop}
\begin{proof} By Lemma \ref{vcstru}, we have a unique maximal finite normal subgroup $F$ of $W$ such that $W / F$ is either isomorphic to $\BZ$ or $\BZ/2\ast\BZ/2$. On the other hand, one easily calculates that  $W/[W,W]$ is a direct product of copies of $\BZ/2$, hence $W / F$ must be isomorphic to $\BZ/2\ast\BZ/2$ in this case. This means that $W/[W,W]$ has at least two copies of $\BZ/2$. This further shows that $A/[A,A]$ contains a copy of $\BZ^2$. Now by \Cref{BVChomology}, it follows that $A$ does not have $b \VCYC$ .
\end{proof}

\textbf{Proof of  \ref{thma}}.  Since Coxeter groups are linear, we know that a Coxeter group has $b \VCYC$ if and only if it is virtually cyclic by \Cref{bVCYC-thm}. Hence we are left to consider the case when the Coxeter group $H$ is finite by \Cref{bfquotient}. In this case, $A$ is an Artin group of spherical type. Calvez and  Wiest showed that in the case that $A$ is an irreducible Artin group of spherical type and it is not right angled,  then $A/Z(A)$ is acylindrically hyperbolic \cite[Theorem 1.3]{CW16}. Since RAAGs are CAT(0) cube groups, by \Cref{bVCYC-thm} and \Cref{bfquotient}, it follows that irreducible Artin group of spherical type do not have $b\VCYC$ unless they are virtually cyclic. When $A$ is not irreducible, it  splits as a direct product of irreducible ones, in particular $H_1(A,\BZ)$ is a free abelian group of rank at least two, and thus reducible Artin groups do not have $b\VCYC$ by \Cref{BVChomology}.

\section{Conjugacy growth of Groups Acting on CAT(0)-Spaces}

In this section, we study the conjugacy growth of groups acting properly on CAT(0) spaces via semi-simple isometries. We will show that when such groups contain $\BZ^2$ as a subgroup, they have strictly faster than linear conjugacy growth. As a consequence, we show that a CAT(0) cube group has at most linear conjugacy growth if and only if it is virtually cyclic. Part of the ideas of the proof given here are drawn from our previous paper \cite[Section 4]{vPW1} where we studied the BVC property for CAT(0) groups.

\subsection{Quick review on Conjugacy growth} Let $G$ be a group with a finite symmetric generating set $S$. We define the word metric on $G$ as follows
$$d_S(g,h) =\min \{n \mid g^{-1}h=s_1s_2\cdots s_n,s_i\in S\}.$$
For any $g\in G$, we define the word length of $g$ via 
$$|g|_S = d_S(e,g),$$
where $e$ is the identity element of $G$. Now given $n>0$, we denote by $B_n(G,S)$ the ball of radius $n$ around the identity element with respect to the word metric. The word growth function is the function that maps $n>0$ to $|B_n(G,S)|$, i.e. the number of elements of distance at most $n$ from the identity.

Now given $n>0$, we can consider the conjugacy classes in $B_n(G,S)$ which we denote by $B_n^c(G,S)$. The conjugacy growth function $g_c(n)$ assigns to $n>0$ the number $|B_n^c(G,S)|$, i.e. the number of conjugacy classes in $G$ which intersect $B_n(G,S)$.  For $f,g \colon \BN \to \BN$, we write $f \preceq g$ if there is some constant $C \in \BN$ such that $f(n) \leq g(Cn)$ for all $n \in \BN$. If $f \preceq g$ and $g \preceq f$, we say that $f$ and $g$ are equivalent and write $f \sim g$. Under this equivalence relation, the growth function and the conjugacy growth function are independent of the choice of generating set. We say that a group has linear (resp. at most linear) conjugacy growth if $g_c(n) \sim n$ (resp. $g_c(n) \preceq n$), and we say that a group has exponential conjugacy growth if $g_c(n) \sim 2^n$ or equivalently if
$$\liminf_{n\rightarrow \infty} \frac{\log|B^c_n(G,S)|}{n} >0.$$
For more information about conjugacy growth, we refer to \cite{GS} and \cite{HuOs}.

\subsection{Conjugacy growth of groups acting on CAT(0) spaces} We now review some standard terminology from \cite{BH}. 

\begin{defn}\cite[II.6.1]{BH} Let $X$ be a metric space and let $g$ be an isometry of $X$. The \define{displacement function} of $g$ is the function $d_g \colon X \rightarrow \BR_+ =\{r\geq 0 \mid r\in \BR\}$ defined by $d_g(x) =d(g x,x)$. The \define{translation length} of $g$ is the number $|g|:= \inf \{d_g (x) \mid x\in X\}$. The set of points where $d_g$ attains this infimum will be denoted by $Min(g)$. More generally, if $G$ is a group acting by isometries on $X$, then $Min(G):=\bigcap_{g\in G} Min(g)$. An isometry~$g$ is called \define{semi-simple} if $Min(g)$ is non-empty. An action of a group by isometries of $X$ is called \define{semi-simple} if all of  its elements are semi-simple.
\end{defn}

The following theorem is known as the Flat Torus Theorem \cite[II.7.1]{BH}.
\begin{thm} \label{flattorusthm}
Let $A$ be a free abelian group of rank $n$ acting properly by semi-simple isometries on a CAT(0) space $X$. Then:

\begin{enumerate}
\item $Min(A) = \bigcap_{\alpha\in A} Min(\alpha)$ is non-empty and splits as a product $Y \times \BE^n$, here $\BE^n$ denotes $\BR^n$ equipped with the standard Euclidean metric.
\item Every element $\alpha \in A$ leaves $Min(A)$ invariant and respects the product decomposition; $\alpha$ acts as the identity on the first factor $Y$ and as a translation on the second factor $\BE^n$.
\item The quotient of each $n$-flat $Y \times \BE^n$ by the action of $A$ is an $n$-torus.
\end{enumerate}
\end{thm}

It is clear that the translation length is invariant under conjugation, i.e. $|h g h^{-1}| = |g|$ for any $g,h \in G$. Moreover, for $g$ semi-simple, we have that $|g^n| = |n| \cdot |g|$ for any $n \in \BZ$, e.g. by the Flat Torus Theorem.

The following theorem is proven by Paul Bernays in his Ph.D. thesis \cite{Be12} which generalizes results of Landau--Ramanujan, see  the introduction of \cite{BMO} for more information. 

\begin{thm}\label{quadratic-form-growth}
Let $f(x,y) = ax^2 +bxy+cy^2$ be a primitive quadratic form over $\BZ$ with nonsquare
discriminant $D = b^2-4ac$, and suppose $f$ is positive in case it is definite. Let $B_f (n)$
be the number of positive integers less than or equal to $n$ which are representable by $f$. Then $B_f(n)$ grows asympotically as fast as $\frac{n}{\sqrt{\ln n}}$, i.e. the limit
$$\lim_{n\to \infty}  \frac{B_f(n)}{\frac{n}{\sqrt{\ln n}} }$$
exists and has a positive value.
\end{thm}

\begin{cor}\label{lattice-plane-growth}
Let $V_1,V_2$ be two linearly independent vectors in the plane $\BR^2$, and let $S(n) $ be the number of elements in
\begin{align*} 
\left\{ ||xV_1+yV_2|| \mid ||xV_1+yV_2||\leq n, x,y\in \BZ  \right\}
\end{align*}
Then $S(n)$ grows asympotically the same as $\frac{n^2}{\sqrt{\ln n}}$ or $n^2$.
\end{cor}

\begin{proof} This is basically \cite[Proposition 1]{MoreeOsburn}. Recall that a two-dimensional lattice $L$ is said to be arithmetic if and only if there exists a real
number $\lambda$ such that $\lambda L$ is isometric to a $\BZ$-submodule of rank two in an
imaginary quadratic number field, otherwise it is said to be non-arithmetic. If the lattice corresponding to $V_1, V_2$ is non-arithmetic, we obtain quadratic growth by \cite[Corollary, p.166]{Kuehnlein}. Otherwise, possibly after scaling, we can apply \Cref{quadratic-form-growth} to obtain a growth rate of $\nicefrac{n^2}{\sqrt{\ln n}}$.
\end{proof}

Recall that a CAT(0) group is a group which acts properly and cocompactly on a CAT(0) space via isometries. 

\begin{thm}\label{sub-Z2-con-growth}
Let $G$ be a CAT(0) group which contains $\BZ^2$ as a subgroup, than the conjugacy growth of $G$ is strictly faster than linear.
\end{thm}
\begin{proof}
Fix a generating set $S = \langle s_1,s_2,\cdots,s_k\rangle $ for $G$ and assume that $G$ acts on the CAT(0) complex $X$ properly and cocompactly. Since $G$ contains $\BZ^2$ and the asymptotics of the conjugacy growth function do not depend on $S$, we can assume that the elements $s_1,s_2$ are a generating set for $H \cong \BZ^2$ in $G$. Now by the Flat Torus Theorem, $H$ acts on a flat plane $P$ inside $X$ via translations. Let $x_0$ be a point in $P$, the translation length of any $h\in H$ can now be calculated easily via $d_X(x_0,hx_0)$ again by the Flat Torus Theorem.  Moreover, 
by the \v{S}varc-Milnor Lemma \cite[I.8.19]{BH}, there exist $C,D>0$ such that for any $g_1,g_2\in G$
\[
d_S(g_1,g_2) \leq Cd_X(g_1x_0,g_2x_0) + D \eqnstop
\]
Now given $n>0$, let $B_H^c(n)$ be the conjugacy classes of elements in $H$ with word length $\leq n$. Since translation length is an invariant of the conjugacy classes and $|g| = d_X(x_0,gx_0)$ for any $g\in H$, we have 
$$ \{ |g| \mid |g|\leq n,g\in H  \} \subseteq B_H^c(nC+D) $$
Now by \Cref{lattice-plane-growth}, $\{ |g| \mid |g|\leq n,g\in H  \}$ grows already faster than linear, hence the same holds for $B_H^c(n)$. Thus the conjugacy growth of $G$ is also faster than linear.
\end{proof}

In order to prove the rest part of \ref{thmb}, we also need the following result which is essentially due to Caprace and Sageev \cite{CS11}.

\begin{lem} \cite[Lemma 4.15]{vPW1} \label{CAT(0)cuberankrigidity} 
Let $G$ be a group which acts on a CAT(0) cube complex $X$ properly and cocompactly via isometries and suppose that $G$ is not virtually cyclic. Then either $G$ contains a rank one isometry or $G$ contains a free abelian subgroup of rank $2$.
\end{lem}

\begin{thm}\label{CAT(0)-cube-con-growth}
A CAT(0) cube group has at most linear  conjugacy growth if and only if it is virtually cyclic.
\end{thm}

\proof 
By \Cref{CAT(0)cuberankrigidity}, we only need to deal with the case that $G$ contains a rank one isometry or $\BZ^2$. If $G$ contains a rank-one isometry, then $G$ is acylindrically hyperbolic \cite{Sis15} and has exponential conjugacy growth \cite[Theorem 1.1]{HuOs}. When $G$ contains $\BZ^2$, the result is implied by \Cref{sub-Z2-con-growth}.
\qed

As an application of \Cref{CAT(0)-cube-con-growth}, we give a new proof of \cite[Corollary 4.16]{vPW1}.

\begin{cor}
A CAT(0) group containing $\BZ^2$ does not have BVC. In particular, a CAT(0) cube group has BVC if and only if it is virtually cyclic.
\end{cor}
\begin{proof}
By \cite[Corollary 2.5]{vPW2},  a CAT(0)  group has BVC if and only if it has at most linear conjugacy growth. Note that BVC implies $b \VCYC$ by definition \cite[Definition 1.10]{vPW2}. Now the Corollary is implied by \Cref{sub-Z2-con-growth} and \Cref{CAT(0)-cube-con-growth}. 
\end{proof}

\begin{rem}
Using \cite[Corollary 6.4]{CS11}, one can actually prove that a CAT(0) cube group has exponential conjugacy growth if and only if it is not virtually abelian.  
\end{rem}

\section{Finiteness of the Classifying Space \texorpdfstring{$\huu BG$}{BG}}

Given a group $G$ and a family of subgroups $\family F$ of $G$ we now want to study the finiteness properties of the classifying space $B_{\family F}(G) = E_{\family F}(G)/G$. Again we shall use the convention that $\u B G = B_{\FIN}(G)$ and $\uu B G = B_{\VCYC}(G)$. As the $G$-homotopy type of $E_{\family F}(G)$ is uniquely determined, so is the homotopy type of $B_{\family F}(G)$. However, if $B_{\family F}(G) \to X$ is some homotopy equivalence to another CW-complex $X$, there need not be a $G$-homotopy equivalence of $G$-CW complexes whose quotient realizes the given map. The situation is different for the trivial family $\family F = \TR$, since $EG$ is the universal cover of $BG$. Thus finiteness conditions of the $G$-CW complex $EG$ are equivalent to finiteness conditions of the CW-complex $BG$. For example, $BG$ has a finite model if and only if $EG$ has a finite model. We shall see below that a corresponding statement fails for finite-dimensionality if we take the family of finite or the family of virtually cyclic subgroups. The following question goes back to \cite[Remark 17]{JuanPinedaLeary} and motivated our study:
\begin{question} \label{uubg-finite-question}
Suppose $\uu BG$ is homotopy equivalent to a finite CW-complex. Is $\uu BG$ necessarily contractible?
\end{question}
In contrast to \Cref{uubg-finite-question}, in the case of the family of finite subgroups, Leary and Nucinkis showed \cite{LearyNucinkis2001} that every connected CW-complex is homotopy equivalent to $\underline{B}G$ for some group $G$. By \cite[Proposition 3]{LearyNucinkis2001} we know that $\pi_1( B_{\family F}(G) ) \cong G/N$ where $N$ is the smallest normal subgroup of $G$ containing all subgroups of $\family F$. In particular, it follows that $\uu B G$ is simply-connected for any group $G$. Then \Cref{uubg-finite-question} is equivalent to the question whether $\uu BG$ is contractible if all homology groups $H_{*}(\uu B G; \Z)$ are finitely generated. 
\Cref{uubg-finite-question} appears to be more difficult than \ref{ConjA} in the sense that our proofs for certain classes of groups depend on the validity of \ref{ConjA}.

In the following let us discuss the question whether $\uu BG$ being homotopy-equivalent to a finite-dimensional complex implies the existence of a finite-dimensional model for $\uu E G$. It is consistent with Zermelo-Fraenkel set theory with axiom of choice (ZFC) that for $G$ locally finite of cardinality $\aleph_n$ the minimal dimension of $\u EG$ is equal to $n + 1$ \cite[Example 5.32]{LueckWeiermann}. 
A lower bound for the dimension of $\u E G$ is provided by the rational cohomological dimension, namely we have $\cd_\Q(G) \leq \fingd(G)$. And it is consistent with ZFC that $\cd_\Q(G) = n+1$. Note that for $G$ locally finite $\VCYC(G) = \FIN(G)$ and $\uu BG = \u BG$ is contractible as we shall see below. In particular, it is consistent with ZFC that the gap between the minimal dimension of a model for $\u BG$ and the minimal dimension of a model for $\u E G$ is arbitrarily large. Actually, it is then also consistent with ZFC that there exists a locally finite group of cardinality $\aleph_{\omega}$ which does not admit a finite-dimensional model for $\uu E G = \u E G$. Summarizing, we have seen that if $\uu BG$ is homotopy-equivalent to a finite CW-complex, it is in general impossible to conclude that $\uu E G$ is finite-dimensional.

\begin{lem} \label{ub-for-virtually-cyclic-subgroups}
Let $V$ be a virtually cyclic group. Then $\u B V$ is contractible if and only if $V$ is finite or nonorientable. If $V$ is orientable, then $\u B V = S^1$. In particular, $H_n(\u BV;\Z) = 0$ for $n \geq 2$ in all cases.
\end{lem}
\begin{proof}
For $V = \Z$, we have of course $\u BV = BV = K(V,1) = S^1$. For $V = D_{\infty} = \Z/2 * \Z/2$ the infinite dihedral group, we get that $\pi_1(\u B D_{\infty}) = 1$ since $D_{\infty}$ is generated by elements of finite order. But more is true: we have $\R$ as a model for $\u E D_{\infty}$, and moreover $\R/D_{\infty} \cong [0,1/2]$.

More generally, if $V$ is an orientable virtually cyclic group, there exists an epimorphism $\pi \colon V \to \Z$ with finite kernel. Then $\R$ serves as a model for $\u E V$ by pulling back the standard $\Z$-action on $\R$ via $\pi$. Thus $\u B V = S^1$. Similarly, if $V$ is nonorientable, then $\u B V$ is contractible.
\end{proof}

\subsection{Locally \texorpdfstring{$\family F$}{F} Groups}

For $G$ a group and $\family F$ a family whose elements are finitely generated subgroups of $G$ we want to give an easy argument here to show that $B_{\family F} G$ is contractible for $G$ locally $\family F$, i.e. a group such that all its finitely generated subgroups lie in the family $\family F$. In \cite[p.~10]{JuanPinedaLeary} Juan-Pineda and Leary noted that $\uu B G$ is contractible for $G$ locally virtually cyclic and provide a proof in the case that $G$ is countable by constructing an explicit model.

Ramras \cite{Ramras} has a given a nice account on functorial models for $E_{\family F}(G)$ resp. $B_{\family F}(G)$. We shall use in the following that $B_{\family F}(G)$ can be viewed as the geometric realization of the nerve of the orbit category $\Or_{\family F}(G)$. Recall that the category $\Or_{\family F}(G)$ has as objects transitive $G$-sets $G/H$ with $H \in \family F$ and as morphisms $G$-maps. One observes that
\[
\Hom(G/H, G/K) \cong \{ g \in G \mid g H g^{-1} \subgroup K \} / \sim
\]
where $g \sim g k$ for all $k \in K$.

Also recall that a category $\mathcal C$ is \define{filtered} if the following two conditions are met:
\begin{enumerate}[(1)]
    \item For any two objects $X, Y$ in $\mathcal C$ there exists an object $Z$ in $\mathcal C$ and morphisms $X \to Z$ and $Y \to Z$.
    \item For two morphisms $f,g \colon X \to Y$ there exists an object $Z$ in $\mathcal C$ and a morphism $h \colon Y \to Z$ such that $hf = hg$.
\end{enumerate}
By a classical result, the nerve of a filtered category is contractible, see e.g. \cite[Corollary 2, p.~93]{QuillenI}.

\begin{prop} \label{locally-f-group-contractible-bf} Let $G$ be a group and let $\family F$ be a family whose elements are  finitely generated subgroups of $G$. If $G$ is locally $\family F$, then $B_{\family F} G$ is contractible.
\end{prop}
\begin{proof}
We verify that the orbit category $\Or_{\family F}(G)$ is filtered. 
\begin{enumerate}[(1)]
\item For two transitive $G$-sets $G/H$, $G/H'$, we note that $\langle H, H' \rangle$ is again in $\family F$ since $H, H'$ are finitely generated and $G$ is locally $\family F$. Then certainly there are $G$-maps $G/H \to G/\langle H, H' \rangle$ and $G/H' \to G/\langle H, H' \rangle$.
\item Let two $G$-maps $\alpha, \beta \colon G/H \to G/K$ be given. These are represented by elements $a \in G$ and $b \in G$, i.e. $\alpha(H) = aK$ and $\beta(H) = bK$. We consider the finitely generated subgroup $L = \langle a, b, K \rangle \in \family F$. Then we let $\pi \colon G/K \to G/L$ be the canonical map. It follows that $\pi \alpha = \pi \beta$.
\end{enumerate}
\end{proof}

In particular, \Cref{locally-f-group-contractible-bf} implies that $\uu B G$ is contractible for all locally virtually cyclic groups $G$.
One might ask whether there are groups $G$ that are not locally virtually cyclic but which nevertheless have a contractible classifying space $\uu BG$. These certainly exist by \cite[Example 4]{JuanPinedaLeary}. However, we cannot find such new examples by only considering filtered categories as the following result shows.

\begin{prop}
Let $G$ be a group and let $\family F$ be a family whose elements are finitely generated subgroups containing all cyclic subgroups. The category $\Or_{\family F}(G)$ is filtered if and only if $G$ is locally $\family F$.
\end{prop}
\begin{proof} The proof of \Cref{locally-f-group-contractible-bf} showed that a locally $\family F$ group has a filtered orbit category. So let us suppose that $\Or_{\family F}(G)$ is a filtered category. Let $a \in G$ be non-trivial element and let $K \in \family F$. Consider the two $G$-maps $\pi \colon G/1 \to G/K, 1 \mapsto K$ and $\alpha \colon G/1 \to G/K, 1 \mapsto aK$. By the second property of filtered categories, there exists some $V \in \family F$ and a $G$-map $\lambda \colon G/K \to G/V$ such that $\lambda \circ \pi = \lambda \circ \alpha$. Let $x \in G$ so that $\lambda(K) = xV$. Since $\lambda$ is a $G$-map, we have $K \subgroup xVx^{-1}$. Moreover, $xV =(\lambda \circ \pi) (1) = (\lambda \circ \alpha)(1) = \lambda (aK) = a x V$ implies that $a \in x V x^{-1}$. Thus $a$ and $K$ both lie in the same subgroup $x V x^{-1} \in \family F$. Now, if $H \subgroup G$ is a finitely generated subgroup, then by an inductive argument we obtain that $H \in \family F$. The assumption that $\family F$ contains all cyclic subgroups is needed in the beginning of the induction.
\end{proof}

\subsection{Lück-Weiermann Construction}

Suppose we have a group $G$ and two families of subgroups $\family F \subseteq \family G$ of $G$. We want to recall a construction due to Lück and Weiermann \cite{LueckWeiermann} that allows us to obtain a model for $E_{\family G}(G)$ from a model for $E_{\family F}(G)$ using a pushout. The spaces that are being attached are classifying spaces for certain generalized normalizer subgroups. We will be interested in the case that $\family F = \FIN$ and $\family G = \VCYC$. As mentioned before, for large classes of groups there exist finite models for the classifying space of proper actions. Our strategy in answering \Cref{uubg-finite-question} can then be outlined as follows: For certain classes of groups we shall obtain $\uu E G$ from $\u E G$ by attaching infinitely many classifying spaces. In a second step we compute the homology of $\uu EG$, at least partially, in the hope that the attached classifying spaces generate enough classes in the homology of $\uu EG$.

To perform the construction we will assume that the set $\mathcal G \setminus \mathcal F$ of subgroups of $G$ is equipped with an equivalence relation $\sim$ that satisfies the following additional properties, which we will refer to as (P):
\begin{equation}\tag{P} \label{equivalence-relation-additional-properties}
   \parbox{0.9\textwidth}{
	
	\begin{enumerate}[(1)]
\item If $H, K \in \mathcal G \setminus \mathcal F$ with $H \subgroup K$, then $H \sim K$.
\item If $H, K \in \mathcal G \setminus \mathcal F$ and $g \in G$, then $H \sim K \Leftrightarrow gHg^{-1} \sim g K g^{-1}$.
\end{enumerate}
   
   }
\end{equation}

\begin{notation} We let $[\mathcal G \setminus \mathcal F]$ denote the set of equivalence classes under the equivalence relation $\sim$ and we denote by $[H] \in [\mathcal G \setminus \mathcal F]$ the equivalence class of an element $H \in \mathcal G \setminus \mathcal F$.
\end{notation}

By property (2) of (P) the $G$-action by conjugation on the set $\family G \setminus \family F$ induces a $G$-action on $[\mathcal G \setminus \mathcal F]$. We then define the subgroup
$$
N_G[H] = \{ g \in G \mid [g^{-1}H g] = [H] \},
$$
which is equal to the isotropy group of $[H]$ under the $G$-action we just explained.

Moreover, we define a family of subgroups of $N_G[H]$ by
$$
\mathcal G[H] = \{ K \subgroup N_G[H] \mid K \in \mathcal G \setminus \mathcal F, [K] = [H] \} \cup ( \mathcal F \cap N_G[H] ).
$$
Note that $\mathcal G[H] \subseteq \mathcal G$.

\begin{defn}[Equivalence relation on $\VCYC \setminus \FIN$]\label{def:strong-equivalence-relation} In the case that $\mathcal F = \FIN$ and $\mathcal G = \VCYC$ we choose the equivalence relation defined by
$$
V \sim W \Leftrightarrow |V \cap W| = \infty,
$$
where $V, W \in \VCYC \setminus \FIN$.
\end{defn}

\begin{thm}[{\cite[Theorem 2.3]{LueckWeiermann}}]\label{theorem:passing-to-larger-families}
Let $\mathcal F \subseteq \mathcal G$ and $\sim$ as above an equivalence relation on $\mathcal G \setminus \mathcal F$ satisfying \eqref{equivalence-relation-additional-properties}.
Let $I$ be a complete system of representatives $[H]$ of the $G$-orbits in $[\mathcal G \setminus \mathcal F]$ under the $G$-action induced by conjugation. Choose arbitrary $N_G[H]$-CW-models for $E_{\mathcal F \cap N_{G}[H]}(N_{G}[H])$ and $E_{\mathcal G[H]}(N_{G}[H])$, and an arbitrary $G$-CW-model for $E_{\mathcal F}(G)$. Define a $G$-CW-complex $X$ by the following cellular $G$-pushout

\begin{center}
\begin{tikzpicture}
  \matrix (m) [matrix of math nodes,row sep=3em,column sep=4em,minimum width=2em, text height=1.5ex, text depth=0.25ex]
  {
     \coprod_{[H] \in I} G \times_{N_G[H]} E_{\mathcal F \cap N_{G}[H]}(N_{G}[H]) & E_{\mathcal F}(G) \\
     \coprod_{[H] \in I} G \times_{N_G[H]} E_{\mathcal G[H]}(N_{G}[H]) & X \\};
  \path[-stealth]
    (m-1-1) edge node [right] {$\coprod_{[H] \in I} \id_{G} \times_{N_G[H]} f_{[H]}$} (m-2-1)
            edge node [above] {$i$} (m-1-2)
    (m-2-1.east|-m-2-2) edge node [below] {}
            node [above] {} (m-2-2)
    (m-1-2) edge node [right] {} (m-2-2)
            ;
\end{tikzpicture}
\end{center}

such that $f_{[H]}$ is a cellular $N_{G}[H]$-map for every $[H] \in I$ and $i$ is an inclusion of $G$-CW-complexes, or such that every map $f_{[H]}$ is an inclusion of $N_G[H]$-CW-complexes for every $[H] \in I$ and $i$ is a cellular $G$-map. Then $X$ is a model for $E_{\mathcal G}(G)$.
\end{thm}

\begin{notation} Let $\mathcal F \subseteq \mathcal G$ be two families of subgroups of $G$. We say that $G$ satisfies $(\CondM_{\mathcal F \subseteq \mathcal G})$ if every subgroup $H \in \mathcal G \setminus \mathcal F$ is contained in a unique subgroup $H_{\max}$ which is maximal in $\mathcal G \setminus \mathcal F$, i.e. if $K \in \mathcal G \setminus \mathcal F$ with $H_{\max} \subgroup K$, then $K = H_{\max}$.

We say that a group $G$ satisfies $(\CondNM_{\mathcal F \subseteq \mathcal G})$ if $G$ satisfies $(\CondM_{\mathcal F \subseteq \mathcal G})$ and every maximal subgroup $H_{\max} \in \mathcal G \setminus \mathcal F$ is a self-normalizing subgroup, i.e. $N_G H_{\max} = H_{\max}$.
\end{notation}

\begin{cor} \label{uueg-from-ueg-condm}
Let $G$ be a group satisfying $(\CondM_{\FIN \subseteq \VCYC})$. Let $\mathcal M$ be a complete system of representatives of the conjugacy classes of maximal infinite virtually cyclic subgroups $V \subgroup G$. Then $\uu EG$ can be obtained by the following cellular $G$-pushout:
\begin{center}
\begin{tikzpicture}
  \matrix (m) [matrix of math nodes,row sep=3em,column sep=4em,minimum width=2em, text height=1.5ex, text depth=0.25ex]
  {
     \coprod_{V \in \mathcal{M}} G \times_{N_G V} \u{E} N_G V & \u{E} G \\
     \coprod_{V \in \mathcal M } G \times_{N_G V} E W_G V & \uu EG \\};
  \path[-stealth]
    (m-1-1) edge node [right] {$\coprod_{F \in \mathcal M} \id_G \times f_V$} (m-2-1)
            edge node [above] {$i$} (m-1-2)
    (m-2-1.east|-m-2-2) edge node [below] {}
            node [above] {} (m-2-2)
    (m-1-2) edge node [right] {} (m-2-2)
            ;
\end{tikzpicture}
\end{center}
Here, $E W_G V$ is viewed as an $N_G V$-CW-complex via the projection map $N_G V \epiarrow W_G V = N_G V / V$, the maps starting from the left upper corner are cellular and one of them is an inclusion of $G$-CW-complexes.
\end{cor}

\begin{cor}\label{corollary:passing-to-larger-families-nm-groups}
Let $G$ be a group satisfying $(\CondNM_{\FIN \subseteq \VCYC})$ and let $\mathcal M$ be a complete system of representatives of the conjugacy classes of maximal infinite virtually cyclic subgroups. Then $\uu EG$ can be obtained via the following cellular $G$-pushout:
\begin{center}
\begin{tikzpicture}
  \matrix (m) [matrix of math nodes,row sep=3em,column sep=4em,minimum width=2em, text height=1.5ex, text depth=0.25ex]
  {
     \coprod_{V \in \mathcal{M}} G \times_{V} \u{E}V & \u{E}G \\
     \coprod_{V \in \mathcal M } G/V & \uu EG \\};
  \path[-stealth]
    (m-1-1) edge node [right] {$\coprod_{V \in \mathcal M} p$} (m-2-1)
            edge node [above] {$i$} (m-1-2)
    (m-2-1.east|-m-2-2) edge node [below] {}
            node [above] {} (m-2-2)
    (m-1-2) edge node [right] {} (m-2-2)
            ;
\end{tikzpicture}
\end{center}
Here, $i$ is an inclusion of $G$-CW-complexes and $p$ is the obvious projection.
\end{cor}

\begin{lem} \label{weiermann-index-set-bvcyc} Let $G$ be a group and let $I$ be a complete set of representatives of conjugacy classes of elements in $[\VCYC \setminus \FIN]$ as in the statement of \Cref{theorem:passing-to-larger-families}. If $G$ has $b\VCYC$, then $I$ is finite. 
\end{lem}
\begin{proof}
Let $V_1, \ldots, V_n$ be witnesses to $b\VCYC$ for $G$. We claim that $|I| \leq n$. For each $V_i$ that is infinite, choose some infinite cyclic subgroup $H_i \subgroup V_i$. If $V \subgroup G$ is some infinite virtually cyclic subgroup, choose some infinite cyclic subgroup $H \subgroup V$. By the $b\VCYC$ property there exists some $g \in G$ such that $H^g \subgroup V_j$ for some $j$. But then $H^g \cap H_j$ is an infinite group, hence $V^g \sim H_j$.
\end{proof}

In the light of \Cref{weiermann-index-set-bvcyc} one has to be cautious that the converse does not hold. First of all, observe that representatives of $[\VCYC \setminus \FIN]$ might as well taken to be infinite cyclic. Then having finitely many conjugacy classes of elements in $[\VCYC \setminus \FIN]$ is equivalent to the statement that there are only finitely many commensurability classes of infinite order elements in the group. Note that by 
\cite[Lemma 4.14]{vPW2} there exists a torsion-free group with only two commensurability classes that fails to have $b\CYC$.

\begin{defn} For a group $G$ we denote by $\Tor(G)$ the subgroup of $G$ which is generated by all elements of finite order.
\end{defn}

As noted in the introduction of this section, we have $\pi_1(\u B G) \cong G / \Tor(G)$.

\begin{rem} Note that $\Tor(G)$ is a characteristic subgroup of $G$. In general, the subgroup $\Tor(G)$ contains elements of infinite order and the quotient $G/\Tor(G)$ is not torsion-free. As an example, consider $G = \Z *_{\Z} D_{\infty} = \langle g,a,b \mid g^2 = ab, a^2 = 1 = b^2 \rangle$. There is an epimorphism $\pi \colon G \to \Z/2$ by killing $a$ and $b$. In an amalgamated product, an element of finite order is conjugate to an element lying in one of the factor groups. Hence $\Tor(G) \subgroup \ker(\pi)$, and thus $g \notin \Tor(G)$ defines an element of order 2 in $G/\Tor(G)$.
\end{rem}

Suppose $\alpha \colon G \to Q$ is a group homomorphism. It induces a map $\u B \alpha \colon \u B G \to \u B Q$ and $\pi_1(\u B \alpha)$ can then be identified with the natural map
\[
G/\Tor(G) \to Q/\Tor(Q)
\]
which is induced by $\alpha$. Thus $H_1(\u B\alpha) \colon H_1(\u BG) \to H_1(\u BQ)$ can be identified with the abelianization of the above map:
$$
H_1(\u B \alpha) \colon (G/\Tor(G))^{\operatorname{ab}} \to (Q/\Tor(Q))^{\operatorname{ab}} 
$$
For an abelian group $A$, we write $A_f = A/\Tor(A)$ for the torsion-free part.

\begin{lem}
Let $G$ be a group satisfying $(\CondM_{\FIN \subseteq \VCYC})$, then there is an exact sequence
\begin{align*}
\ldots \to \bigoplus_{V \in \mathcal M} H_2(\u B N_G V) \to & H_2(\u BG) \oplus \bigoplus_{V \in \mathcal M} H_2(B W_G V) \to H_2(\uu BG) \to  \\
& \to \bigoplus_{V \in \mathcal M} (N_G V/\Tor(N_G V))^{\abel}  \to (G/\Tor(G))^{\abel} \oplus \bigoplus_{V \in \mathcal M} (W_G V)^{\abel} \to 0
\end{align*}
\end{lem}
\begin{proof} The long exact sequence arises as the Mayer-Vietoris sequence for the pushout obtained from \Cref{uueg-from-ueg-condm}.
\end{proof}

\begin{lem} \label{second-homology-nm-fin-vcyc}
Let $G$ be a group satisfying $(\operatorname{NM}_{\FIN \subset \VCYC})$ and let $\mathcal M$ be a complete system of representatives of the conjugacy classes of maximal infinite virtually cyclic subgroups. Then there is an exact sequence
\begin{align*}
0 \to H_2(\u BG) \to H_2(\uu BG) \to \bigoplus_{V \in \mathcal M}
(V/\Tor(V))^{\abel}  \to (G/\Tor(G))^{\abel} \to 0 \eqnstop
\end{align*}
Here, $H_2(\u BG) \to H_2(\uu BG)$ is induced by the canonical map $\u BG \to \uu BG$ and the inclusions $V \to G$ for $V \in \mathcal M$ induce the other map. For $n > 2$, the canonical map $H_n(\u BG) \to H_n(\uu BG)$ is an isomorphism. Moreover, note that
\[
\bigoplus_{V \in \mathcal M} (V/\Tor(V))^{\abel} \cong \bigoplus_{V \in \mathcal M^{o}} \Z \eqncomma
\]
where $\mathcal M^o$ denotes the subset of $\mathcal M$ consisting only of orientable infinite virtually cyclic subgroups.
\end{lem}
\begin{proof} By taking the $G$-quotient of the pushout of \Cref{corollary:passing-to-larger-families-nm-groups} we obtain the following long exact sequence
\begin{align*}
\ldots \to \bigoplus_{V \in \mathcal M} H_2(\u B V) \to H_2(\u BG) \to H_2(\uu BG) \to \bigoplus_{V \in \mathcal M} H_1(\u BV)  \to H_1(\u BG) \to 0
\end{align*}
The sequence is exact at the right, since $\uu BG$ is simply-connected, so $H_1(\uu BG) = 0$. By \Cref{ub-for-virtually-cyclic-subgroups}, $H_n(\u B V) = 0$ for all virtually cyclic groups $V$ for $n \geq 2$. 
\end{proof}

By \cite[Example 3.6]{LueckWeiermann} a hyperbolic group satisfies the condition $(\operatorname{NM}_{\FIN \subseteq \VCYC})$. Moreover, by \cite[Theorem 13]{JuanPinedaLeary} there are infinitely many conjugacy classes of orientable maximal infinite virtually cyclic subgroups. Hence we obtain:

\begin{cor} Let $G$ be a non-elementary hyperbolic group. Then $H_2(\uu BG)$ contains a free abelian group of infinite rank.
\end{cor}

This was already shown by Juan-Pineda and Leary \cite[Corollary 16]{JuanPinedaLeary}, albeit with a slightly different proof.

\subsection{Abelian and Poly-\texorpdfstring{$\Z$}{Z}-Groups}

Juan-Pineda and Leary computed in \cite[Example 3]{JuanPinedaLeary} that $H_2( \uu B \Z^2)$ and $H_3(\uu B \Z^2)$ are free abelian of infinite rank using an explicit model that they constructed. Let us consider more generally $G = \Z^n$ for $n \geq 2$. 
Using \Cref{uueg-from-ueg-condm} and the fact that $N_G V = G = \Z^n$, $W_G V = G/V \cong \Z^{n-1}$ for $V$ maximal infinite cyclic, we obtain the following long exact sequence
\[
0 \to H_{n+1}(\uu BG) \to \bigoplus_{V \in \mathcal M} H_n(B G) \to H_n(B G) \oplus \bigoplus_{V \in \mathcal M} H_n(B \Z^{n-1}) \to \ldots \eqncomma
\]
where $\mathcal M$ denotes the set of maximal infinite cyclic subgroups of $G$. Note that $\mathcal M$ is infinite and since $H_n(B \Z^{n-1}) = 0$ and $H_n(B \Z^n) \cong \Z$, it follows that $H_{n+1}(\uu B G)$ is a free abelian group of infinite rank. This also implies that $\vcycgd(\Z^n) \geq n + 1$. In fact, by \cite[Example 5.21]{LueckWeiermann} we have $\vcycgd(\Z^n) = n + 1$. For $G$ finitely generated abelian we have $G \cong \Z^n \oplus T$ with $T$ finite abelian. It follows that $H_{n+1}(\uu BG)$ contains a free abelian subgroup of infinite rank as a direct summand whenever $G$ is not virtually cyclic.

\begin{prop} \label{abelian-finite-uubg}
Let $G$ be an abelian group that is not locally virtually cyclic. Then $H_2(\uu BG)$ is not finitely generated.
\end{prop}
\begin{proof} First note that an abelian group is locally virtually cyclic if and only if it does not contain a copy of $\Z^2$ as a subgroup. In particular, it follows that the complete set $I$ of representatives of elements in $[\VCYC \setminus \FIN]$ is infinite. Since $G$ is abelian we have $N_G[V] = G$ for any virtually cyclic $V$. By \Cref{theorem:passing-to-larger-families} there exists a $G$-pushout 
\begin{center}
\begin{tikzpicture}
  \matrix (m) [matrix of math nodes,row sep=3em,column sep=4em,minimum width=2em, text height=1.5ex, text depth=0.25ex]
  {
     \coprod_{V \in I} G \times_{G} \u E G & \u E G \\
     \coprod_{V \in I} G \times E_{\VCYC[V]}G & \uu E G \\};
  \path[-stealth]
    (m-1-1) edge node [right] {$\coprod_{V \in I} \id_G \times_G f_{[V]}$} (m-2-1)
            edge node [above] {$i$} (m-1-2)
    (m-2-1.east|-m-2-2) edge node [below] {}
            node [above] {} (m-2-2)
    (m-1-2) edge node [right] {} (m-2-2)
            ;
\end{tikzpicture}
\end{center}
Here, $\VCYC[V] = \{ K \subgroup G \mid K \in \VCYC(G) \text{ and } |K \cap V| = \infty \} \cup \FIN$. After taking the quotient by $G$ we obtain the following part of the Mayer-Vietoris sequence:
\[
\ldots \to H_2(\uu BG) \to \bigoplus_{V \in I} H_1(\u B G) \to H_1(\u B G) \oplus \bigoplus_{V \in I} H_1(B_{\VCYC[V]}G) \to H_1(\uu BG) = 0
\]
As before, let $G_f = G/\Tor(G)$ denote the torsion-free quotient of $G$.
Now the last non-trivial map in the long exact sequence can be identified with
\[
\theta \colon \bigoplus_{V \in I} G_f \to G_f \oplus \bigoplus_{V \in I} G_f/N_V
\]
where $N_V = \langle K \mid K \text{ cyclic and } |K \cap V| = \infty \} \subgroup G_f$ and given by the sum of $\id_G$ and the canonical projections. Then $\ker(\theta) = \{ (g_V)_{V \in I} \in \bigoplus_{V \in I} N_V \mid \sum_{V \in I} g_V = 0 \}$ which is not finitely generated. Then $H_2(\uu BG)$, which surjects onto $\ker(\theta)$ cannot be finitely generated.
\end{proof}

We see in particular that an abelian group $G$ has a contractible classifying space $\uu BG$ if and only if it is locally virtually cyclic.
It is also worthwhile to note that a torsion-free locally cyclic group is isomorphic to a subgroup of the rational numbers $\Q$, see e.g \cite[Chapter VIII, Section 30]{Kurosh1955}.

We call a group $G$ \define{poly-$\Z$} if there exists a chain of subgroups
$1 = G_0 \subgroup G_1 \subgroup G_2 \subgroup \ldots \subgroup G_n = G$
such that $G_i \unlhd G_{i+1}$ and $G_{i+1}/G_i$ is infinite cyclic for all $i = 0, 1, \ldots, n-1$.
Note that poly-$\Z$ groups do not necessarily satisfy the condition $\CondM_{\FIN \subseteq \VCYC}$. An example is already provided by the non-trivial extension $\Z \rtimes \Z$, see \cite[Example 3.7]{LueckWeiermann}. For a poly-$\Z$ group $G$ we know that the cohomological dimension $\cd(G)$ is given by $\cd(G) = \max \{ i \mid H_i(G; \Z/2) \neq 0 \}$, see e.g. \cite[Example 5.26]{LueckSurveyClassifyingSpaces}.

\begin{prop}
Let $G$ be a poly-$\Z$ group that is not infinite cyclic. Then there is some $n$ such that $H_n(\uu B G; \Z/2)$ is not finitely generated. 
\end{prop}
\begin{proof}
By \cite[Example 5.26]{LueckSurveyClassifyingSpaces} we know that there exists a finite model for $\u E H$ for any virtually poly-$\Z$ group $H$. In \cite[Theorem 5.13]{LueckWeiermann} a model of minimal dimension for $\uu E G$ is being constructed. In the course of this proof one obtains the following pushout, where the index set $I$ runs over certain infinite cyclic subgroups of $G$:
\begin{center}
\begin{tikzpicture}
  \matrix (m) [matrix of math nodes,row sep=3em,column sep=4em,minimum width=2em, text height=1.5ex, text depth=0.25ex]
  {
     \coprod_{C \in I} G \times_{N_G C} \u E N_G C & \u E G \\
     \coprod_{C \in I} G \times_{N_G C} \u E W_G C & \uu E G \\};
  \path[-stealth]
    (m-1-1) edge node [right] {$\coprod_{C \in I} \id_G \times_{N_G C} f_C$} (m-2-1)
            edge node [above] {$i$} (m-1-2)
    (m-2-1.east|-m-2-2) edge node [below] {}
            node [above] {} (m-2-2)
    (m-1-2) edge node [right] {} (m-2-2)
            ;
\end{tikzpicture}
\end{center}
Here, $i$ is an inclusion of $G$-CW complexes and $f_C$ is a cellular $N_G C$-map for every $C \in I$. Observe that $I$ has to be infinite. Otherwise, we would obtain a classifying space $\uu E G$ of finite type since $N_G C$ and $W_G C$ are virtually poly-$\Z$. But this is impossible by \Cref{bVCYC-thm} since $G$ is solvable but not virtually cyclic.
By taking the quotient by $G$ one obtains the following pushout:

\begin{center}
\begin{tikzpicture}
  \matrix (m) [matrix of math nodes,row sep=3em,column sep=4em,minimum width=2em, text height=1.5ex, text depth=0.25ex]
  {
     \coprod_{C \in I} \u B N_G C & \u B G \\
     \coprod_{C \in I} \u B W_G C & \uu B G \\};
  \path[-stealth]
    (m-1-1) edge node [right] {} (m-2-1)
            edge node [above] {} (m-1-2)
    (m-2-1.east|-m-2-2) edge node [below] {}
            node [above] {} (m-2-2)
    (m-1-2) edge node [right] {} (m-2-2)
            ;
\end{tikzpicture}
\end{center}
Of course, as $G$ is torsion-free, we have $\u BG = BG$ and $\u B N_G C = B N_G C$. From the pushout, we obtain the following Mayer-Vietoris sequence, suppressing the coefficient group $\Z/2$ in the notation:
\[
\ldots \to H_{k+1}(\uu BG) \to H_k\left(\coprod_{C \in I} B N_G C \right) \to H_k(BG) \oplus H_k \left( \coprod_{C \in I} \u B W_G C \right) \to \ldots 
\]
We also observe that $\gd(G) = \vcd(G) = \cd(G)$, $\gd(N_G C ) = \cd(N_G C )$ and $\fingd( W_G C) = \cd(N_G C ) - 1$, the proof of which can be found in the proof of \cite[Theorem 5.13]{LueckWeiermann} as well. In particular, we see that the homology groups of all spaces appearing in the pushout, will vanish in large enough degrees. Now, let $k$ be the largest integer such that there are infinitely many $C \in I$ with $\gd( N_G C) = k$ and with only finitely many $C \in I$ with $\gd(N_G C) = k+1$. Then there are only finitely many $C \in I$ with $\gd( \u B W_G C) \leq k$. Observe that $H_k( N_G C; \Z/2) \neq 0$ for infinitely many $C$. As $H_k(BG)$ is finite, the above exact sequence shows that $H_{k+1}(\uu BG)$ cannot be finitely generated.
\end{proof}

As a corollary we obtain an affirmative answer to \Cref{uubg-finite-question} for the class of poly-$\Z$-groups.


\printbibliography

\end{document}